\documentclass{amsart}

\usepackage{amssymb}
\usepackage[hidelinks]{hyperref}
\usepackage[textsize=footnotesize]{todonotes}
\usepackage{tikz}
\usetikzlibrary{cd}
\usepackage{float}

\theoremstyle{plain}
\newtheorem{theorem}{Theorem}[section]
\newtheorem{lemma}[theorem]{Lemma}
\newtheorem{corollary}[theorem]{Corollary}
\newtheorem{proposition}[theorem]{Proposition}

\newtheorem{remark}[theorem]{Remark}

\theoremstyle{definition}
\newtheorem{definition}[theorem]{Definition}

\newtheorem{problem}{Problem}

\theoremstyle{remark}

\newcommand{\cA}{\mathcal{A}}

\newcommand{\cD}{\mathcal{D}}
\newcommand{\cE}{\mathcal{E}}
\newcommand{\cF}{\mathcal{F}}

\newcommand{\cI}{\mathcal{I}}
\newcommand{\I}{\cI}
\newcommand{\cJ}{\mathcal{J}}
\newcommand{\J}{\cJ}
\newcommand{\cM}{\mathcal{M}}
\newcommand{\cN}{\mathcal{N}}
\newcommand{\cP}{\mathcal{P}}

\newcommand{\cK}{\mathcal{K}}
\newcommand{\K}{\cK}

\newcommand{\cU}{\mathcal{U}}

\newcommand{\continuum}{\mathfrak{c}}
\newcommand{\cc}{\continuum}
\newcommand{\bnumber}{\mathfrak{b}}
\newcommand{\bb}{\bnumber}
\newcommand{\dnumber}{\mathfrak{d}}

\newcommand{\unumber}{\mathfrak{u}}

\newcommand{\uu}{\mathfrak{u}}
\newcommand{\pp}{\mathfrak{p}}
\newcommand{\dd}{\dnumber}
\newcommand{\cf}{\mathrm{cf}} 
\DeclareMathOperator{\add}{add}
\DeclareMathOperator{\cl}{cl}
\newcommand{\fin}{\mathrm{Fin}}

\newcommand{\nwd}{\mathrm{NWD}}  
\newcommand{\mz}{\mathrm{NULL}}  
\newcommand{\conv}{\mathrm{CONV}} 
\newcommand{\ctbl}{\mathrm{CTBL}}

\DeclareMathOperator{\lh}{lh}

\begin{document}


\title{More on yet another ideal version of the bounding number}


\author[Adam Kwela]{Adam Kwela}
\address[Adam Kwela]{Institute of Mathematics\\ Faculty of Mathematics\\ Physics and Informatics\\ University of Gda\'{n}sk\\ ul.~Wita  Stwosza 57\\ 80-308 Gda\'{n}sk\\ Poland}
\email{Adam.Kwela@ug.edu.pl}
\urladdr{https://mat.ug.edu.pl/~akwela}


\date{\today}






\begin{abstract}
This is a continuation of the paper [J. Symb. Log. 87 (2022), 1065--1092]. For an ideal $\I$ on $\omega$ we denote $\mathcal{D}_\I=\{f\in\omega^\omega: f^{-1}[\{n\}]\in\I \text{ for every $n\in \omega$}\}$ and write $f\leq_\I g$ if $\{n\in\omega:f(n)>g(n)\}\in\I$, where $f,g\in\omega^\omega$. 

We study the cardinal numbers $\bb(\geq_\I\cap (\cD_\I \times \cD_\I))$ describing the smallest sizes of subsets of $\mathcal{D}_\I$ that are unbounded from below with respect to $\leq_\I$. 

In particular, we examine the relationships of $\bb(\geq_\I\cap (\cD_\I \times \cD_\I))$ with the dominating number $\dd$. We show that, consistently, $\bb(\geq_\I\cap (\cD_\I \times \cD_\I))>\dd$ for some ideal $\I$, however $\bb(\geq_\I\cap (\cD_\I \times \cD_\I))\leq\dd$ for all analytic ideals $\I$. Moreover, we give example of a Borel ideal with $\bb(\geq_\I\cap (\cD_\I \times \cD_\I))=\add(\cM)$. 
\end{abstract}


\maketitle


\section{Introduction}

Let $\I$ be an ideal on $\omega$. For $f,g\in\omega^\omega$ write $f \leq_\I  g$ if $\{n\in\omega:f(n)>g(n)\}\in\I$. Moreover, denote $\mathcal{D}_\I=\{f\in\omega^\omega: f^{-1}[\{n\}]\in\I \text{ for every $n\in \omega$}\}$. Let $\bb(\geq_\I\cap (\cD_\I \times \cD_\I))$ (for simplicity denoted also by $\bb(\I)$) be the smallest size of a set in $\mathcal{D}_\I$ not bounded from below with respect to the order $\leq_\I$ by any member of $\mathcal{D}_\I$:
$$\bb(\I)=\min\left\{|\cF|:\cF\subseteq\cD_\I\land\neg(\exists_{g\in \cD_\I}\, \forall_{f\in \cF}\, g\leq_\I f\right\}.$$

The cardinal $\bb(\I)$ in the case of maximal ideals has been deeply studied by Canjar in 1980s in the context of smallest sizes of cofinal and coinitial subsets in ultrapowers $\omega^\omega/\cU$ ordered by $\leq_{\I}$, where $\cU=\I^\star$ (see~\cite{Canjar2,Canjar,CanjarPhD}). In the case of Borel ideals, we have extensively examined $\bb(\I)$ in~\cite{nasza}. 

This research is partially motivated by the study of ideal-QN-spaces. A topological space $X$ is a QN-space if it does not distinguish pointwise and quasi-normal convergence of sequences of real-valued continuous functions defined on $X$. The research on QN-spaces was initiated by Bukovsk\'{y}, Rec\l{}aw and Repick\'{y} in \cite{MR1129696}, who proved that the smallest size of non-QN-space is equal to the bounding number $\bnumber$. Studies of QN-spaces have been continued in papers \cite{MR2463820,MR2778559,MR2294632,MR1129696,MR1815270,MR1477547,MR1800160,MR2280899,MR2881299}. Ideal variants of QN-spaces were introduced in \cite{MR3038073} and studied in \cite{MR3622377,MR3784399,MR3924519,MR3423409,Repicky-1,Repicky-2}. For a given ideal $\I$, the cardinal number $\bb(\I)$ characterizes the smallest size of a space which is not $\I$QN (see~\cite{Repicky-1}).

The paper is organized as follows.

In Section~\ref{sec:preliminaries} we collect basics about ideals on $\omega$ and some known facts about the cardinal numbers $\bb(\I)$. In particular, we recall a very useful combinatorial characterization of $\bb(\I)$ from~\cite{nasza}, which we use almost exclusively in the rest of the paper.

In the remaining part of the paper we answer two natural questions concerning $\bb(\I)$. 

The first question is about possible values of $\bb(\I)$ for Borel ideals: in~\cite{nasza} it is shown that there are $\bf{\Sigma^0_2}$ ideals $\I$ with $\bb(\I)=\aleph_1$ provable in ZFC (\cite[Corollary~7.3 and Theorem~7.4]{nasza}) as well as $\bf{\Sigma^0_2}$ ideals $\I$ with $\bb(\I)=\bnumber$ provable in ZFC (\cite[Example~5.15]{nasza}). However, it was unknown whether $\aleph_1$ and $\bb$ in the above can be replaced by other cardinals. In Section~\ref{sec:toprep} we study $\bb(\I)$ in the case of ideals with topological representation (such ideals were introduced by Sabok and Zapletal in~\cite{SabokZapletal}) obtaining a $\bf{\Pi^0_3}$ ideal with $\bb(\I)=\add(\cM)$ provable in ZFC. 

The second question concerns comparison of $\bb(\I)$ with the dominating number $\dd$. In general, $\aleph_1\leq\bb(\I)\leq\cc$, for every ideal $\I$ (\cite[Theorem~4.2]{nasza}). Moreover, the upper bound can be improved to $\bb(\I)\leq\bb$ in the case of $\bf{\Pi^0_4}$ ideals $\I$ (\cite[Corollary~6.9]{nasza}; in fact, as shown in \cite[Theorem~6.8]{nasza}, this holds even for a larger class of all so-called co-analytic weak P-ideals). On the other hand, Canjar proved in~\cite{Canjar,Canjar2,CanjarPhD} that under $\dd=\cc$ there is always a maximal ideal $\I$ with $\bb(\I)=\cf(\dd)$ and that in the model obtained by adding $\lambda$ Cohen reals to a model of GCH, $\bb(\I)$ can be equal to any regular cardinal between $\aleph_1$ and $\lambda=\dd$, for some maximal ideal $\I$. However, it was unknown if it is consistent to have an ideal $\I$ with $\bb(\I)>\dd$. We show that the answer is affirmative (Section~\ref{sec:>d}), however $\bb(\I)\leq\dd$ for every analytic ideal $\I$ (Section~\ref{sec:analytic}). The latter result uses ideas developed by Debs and Saint Raymond in~\cite{Debs}, providing a new method of proving statements about analytic ideals.


\section{Preliminaries}
\label{sec:preliminaries}

By $\omega$ we denote the set of all natural numbers. We identify  a natural number $n$ with the set $\{0, 1,\dots , n-1\}$.
We write $A\subseteq^\star B$ if $A\setminus B$ is finite.
For a set $A$ and a finite or infinite cardinal number  $\kappa$, we write $[A]^{\kappa} =\{B\subseteq A: |B|=\kappa\}$ and $[A]^{<\kappa} =\{B\subseteq A: |B|<\kappa\}$. Moreover, by $\cf(\kappa)$ we denote the cofinality of $\kappa$.


\subsection{Ideals and \texorpdfstring{$\sigma$}{}-ideals}

An \emph{ideal on a set $X$} is a family $\I\subseteq\cP(X)$ satisfying the following properties:
\begin{enumerate}
\item[(i)] if $A,B\in \I$ then $A\cup B\in\I$;
\item[(ii)] if $A\subseteq B$ and $B\in\I$ then $A\in\I$;
\item[(iii)] $\I$ contains all finite subsets of $X$;
\item[(iv)] $X\notin\I$.
\end{enumerate}
Note that, thanks to item (iii), $\bigcup\I=X$ -- we will use this observation several times in Section \ref{sec:analytic}. A \emph{$\sigma$-ideal on $X$} is an ideal satisfying one additional property:
\begin{enumerate}
\item[(v)] if $\langle A_n:n\in\omega\rangle\in \I^\omega$ then $\bigcup_{n\in\omega}A_n\in\I$.
\end{enumerate}
Note that, unlike most Authors, in our paper every $\sigma$-ideal on $X$ already contains all countable subsets of $X$ and has to be a proper subset of $\cP(X)$ (i.e., cannot be equal to $\cP(X)$). In this paper, by an \emph{ideal} we mean ideal on some countable set, while all $\sigma$-ideals will be subsets of $\cP(2^\omega)$. 

The ideal of all finite subsets of $\omega$ is denoted by $\fin$. We say that an ideal $\I$ on $X$ is \emph{maximal} if $\I\subseteq\J$ implies $\I=\J$, for every ideal $\J$ on $X$.

If $\I$ and $\J$ are ideals on $X$ and $Y$, respectively, then we say that $\I$ and $\J$ are \emph{isomorphic}, if there is a bijection $f:Y\to X$ such that:
$$A\in\I\ \Leftrightarrow\ f^{-1}[A]\in\J,$$
for every $A\subseteq X$. It is easy to see that many properties of ideals are preserved under isomorphisms of ideals.

For $\cA\subseteq\cP(X)$ we write:
$$\I(\cA)=\left\{B\subseteq X: B\subseteq^\star \bigcup\cA'\text{ for some }\cA'\in[\cA]^{<\omega}\right\}.$$
If $X\notin\I(\cA)$, then $\I(\cA)$ is an ideal, which we call \emph{the ideal generated by $\cA$}. Note that in this case $\I(\cA)$ is the smallest ideal containing $\cA$.

By identifying subsets of a countable set $X$ with their characteristic functions, we can equip $\cP(X)$ with the topology of the Cantor space $\{0,1\}^X$ and therefore assign descriptive complexity to ideals on $X$. In particular, an ideal $\I$ is analytic if $\I$ is analytic as a subset of the space $\{0,1\}^X$. 

By $\cM$ ($\cN$) we denote the $\sigma$-ideals of meager (null, respectively) subsets of $2^\omega$. If $\cA\subseteq\cP(2^\omega)$ and $2^\omega$ cannot be covered by countably many members of $\cA$, then by $\sigma\cA$ we denote the \emph{$\sigma$-ideal generated by $\cA$}, i.e.: 
$$\sigma\cA=\left\{B\subseteq 2^\omega: \exists_{\langle A_n\rangle\in \cA^\omega}\, B\subseteq\bigcup_{n\in\omega}A_n\right\}.$$
For a $\sigma$-ideal $I$ on $2^\omega$ let $\overline{I}$ be the family of all compact subsets of $2^\omega$ belonging to $I$. We will say that \emph{$I$ is generated by compact sets} if $I=\sigma\overline{I}$. In particular, $[2^\omega]^{<\omega}$, $\cM$ and $\sigma\overline{\cN}$ (i.e., the $\sigma$-ideal on $2^\omega$ generated by compact null subsets of $2^\omega$) are generated by compact sets.

\subsection{Some cardinal invariants}

In our paper we will need the following cardinal invariants:
\begin{itemize}
\item the \emph{pseudointersection number}:
\begin{align*}
\pp =\min\left\{|\mathcal{A}|:\mathcal{A}\subseteq[\omega]^\omega\ \land \left(\forall_{\mathcal{A}_0\in[\mathcal{A}]^{<\omega}}\, \bigcap\mathcal{A}_0\neq\emptyset\right)\land\left( \forall_{S\in[\omega]^\omega}\, \exists_{\textit{A}\in\mathcal{A}}\, |S\setminus\textit{A}|=\omega\right)\right\};
\end{align*}
\item the \emph{additivity of the $\sigma$-ideal of meager subsets of $2^\omega$} is given by:
$$\add(\cM)=\min\left\{|\cA|:\cA\subseteq\cM\land\bigcup\cA\notin\cM\right\};$$
\item the \emph{bounding number} $\bb$, which is the smallest size of an \emph{$\leq_\fin$-unbounded subset of $\omega^\omega$}, that is:
$$\bb = \min\left\{|\cF|: \cF\subseteq\omega^\omega\land \neg(\exists_{g\in \omega^\omega}\, \forall_{f\in \cF}\, f\leq_\fin g\right\};$$
\item the \emph{dominating number} $\dd$, which is the smallest size of a \emph{$\leq_\fin$-dominating subset of $\omega^\omega$}, that is:
$$\dd = \min\left\{|\cF|: \cF\subseteq\omega^\omega\land \forall_{g\in \omega^\omega}\, \exists_{f\in \cF}\, g\leq_\fin f\right\};$$
\item the \emph{ultrafilter number} $\uu$, which is the smallest size of a family generating a maximal ideal.
\end{itemize}

It is known that: 
$$\aleph_1\leq\pp\leq\add(\cM)\leq\bb\leq\dd\leq\cc$$ 
(see \cite[Subsection 9.2]{MR2778559}). Moreover, $\bb\leq\uu\leq\cc$, however both $\dd<\uu$ and $\uu\leq\dd$ are consistent (see \cite{MR2768685}).

The cardinal $\add(\cM)$ is a particular case of a more general invariant: if $I$ is a $\sigma$-ideal on $2^\omega$, then:
$$\add(I)=\min\left\{|\cA|:\cA\subseteq I\land\bigcup\cA\notin I\right\}.$$
Note that $\add(I)\geq\aleph_1$. Moreover, if $I$ is generated by compact sets, then:
$$\add(I)=\min\left\{|\cA|:\cA\subseteq \overline{I}\land\bigcup\cA\notin I\right\}.$$

\subsection{A characterization of \texorpdfstring{$\bb(\I)$}{}}

Let $\I$ be an ideal on a set $X$.
By $ \widehat{\cP}_\I$ we will denote the family of all sequences $\langle A_n : n\in\omega\rangle\in [\cP(X)]^\omega$ such that $A_n\in \I$ for all $n\in\omega$ and $A_n\cap A_k=\emptyset$ whenever $n\neq k$. By $ \cP_\I$ we will denote the family of all sequences $\langle A_n : n\in\omega\rangle \in \widehat{\cP}_\I$ such that $\bigcup\{A_n: n\in \omega\} = X$.

By \cite[Theorem 3.10]{nasza}, the studied cardinal number $\bb(\I)$ has a useful combinatorial characterization, which we will use almost exclusively in the rest of the paper without any reference:
$$\bb(\I) = \min \left\{|\cE|:\cE\subseteq\widehat{\cP}_\I \land \forall_{\langle A_n\rangle\in\cP_\I}\, \exists_{\langle E_n\rangle\in\cE} \, \bigcup_{n\in\omega}\left(A_n\cap \bigcup_{i\leq n}E_i\right)\notin\I\right\}.$$

The following immediate observation will simplify some of our considerations. 

\begin{remark}
\label{rem-zawieranie}
If $\langle B_n:n\in\omega\rangle$ is a partition of a set $X$ (that is $\bigcup_{n\in\omega}B_n=X$ and $B_n\cap B_m=\emptyset$ for all $n\neq m$) and $\langle A_i:i\in\omega\rangle$ is any sequence of subsets of $X$, then:
$$\bigcup_{i\in\omega}\left(A_i\setminus \bigcup_{n<i}B_n\right)=\bigcup_{n\in\omega}\left(B_n\cap\bigcup_{i\leq n}A_i\right)$$ 
\end{remark}


\section{Ideals with topological representation}
\label{sec:toprep}

In this Section we will deal with ideals having a topological representation in the sense of Sabok and Zapletal (see \cite{SabokZapletal}).

Let $X$ be a separable metrizable space with a countable dense set $D$ and $I$ be a $\sigma$-ideal on $X$ (recall that in our paper every $\sigma$-ideal on $X$ contains all singletons and is a proper subset of $\cP(X)$). Following \cite{SabokZapletal}, we define an ideal on $D$ by: 
$$\J_I=\{A\subseteq D: \cl_X(A)\in I\}$$
(here $\cl_X(A)$ denotes the closure of the set $A$ in the space $X$). We say that an ideal \emph{has a topological representation} if it is isomorphic to some $\J_I$ as above. In such case, we say that it is \emph{represented on $X$ by $I$}.

Note that $\J_I=\J_{\sigma\overline{I}}$, since $\J_I$ depends only on closed members of $I$. By \cite[Proposition 2.1]{KwelaSabok}, two ideals represented by the same $\sigma$-ideal $I$, but defined on different countable dense subsets of $X$, are isomorphic. Moreover, each ideal with topological representation can be represented on the Cantor space $2^\omega$ by a $\sigma$-ideal generated by some family of compact nowhere dense sets (\cite[Corollary 1.3]{KwelaSabok}). Finally, every analytic ideal with topological representation is ${\bf{\Pi^0_3}}$, but not ${\bf{\Sigma^0_2}}$ (\cite[Theorem 1.4]{KwelaSabok}).

The basic examples of ideals having topological representation are $\nwd=\J_\cM$, $\mz=\J_{\sigma\overline{\cN}}$ and $\ctbl=\J_{[2^\omega]^{\leq\omega}}$. First two of them were introduces by Farah and Solecki in \cite{FarahSolecki}.

Before calculating the invariant $\bb(\I)$ for ideals with topological representation, we need to introduce two notions.

\begin{definition}
For a $\sigma$-ideal $I$ on $2^\omega$ we define: 
$$\add'(I)=\min\left\{|\cA|:\ \cA\subseteq\overline{I}\land\forall_{\langle B_n\rangle\in\overline{I}^\omega}\, \exists_{A\in\cA}\, \forall_{n\in\omega}\, A\setminus B_n\neq\emptyset\right\}.$$
\end{definition}

The coefficient $\add'(I)$ has been studied implicitly by Fremlin in \cite[Definition 22A]{Fremlin} and by Kankaanp\"{a}\"{a} in \cite{MR3032424}. Note that $\add'(I)=\add'(\sigma\overline{I})$, as $\add'(I)$ depends only on compact members of $I$.

\begin{proposition}
\label{prop-toprep}
Let $I$ be a $\sigma$-ideal on $2^\omega$ generated by some family of compact sets. Then
$$\aleph_1\leq\add'(I)\leq\add(I).$$
\end{proposition}

\begin{proof}
To show that $\aleph_1\leq\add'(I)$, fix any countable $\cA\subseteq\overline{I}$. Then $\cA=\{B_n: n\in\omega\}$. Consider now the sequence $\langle B_n:n\in\omega\rangle\in\overline{I}^\omega$. Clearly, for each $A\in\cA$ there is $n\in\omega$ with $A=B_n$, so $A\setminus B_n=\emptyset$. Thus, $\aleph_1\leq\add'(I)$.

Now we show that $\add'(I)\leq\add(I)$. Since $I$ is generated by some family of compact sets, we have:
\begin{align*}
    \add(I) & =\min\left\{|\cA|:\ \cA\subseteq I\land\bigcup\cA\notin I\right\}\\
    & =\min\left\{|\cA|:\ \cA\subseteq \overline{I}\land\bigcup\cA\notin I\right\}\\
    & =\min\left\{|\cA|:\ \cA\subseteq\overline{I}\land\forall_{\langle B_n\rangle\in\overline{I}^\omega}\, \bigcup\cA\setminus\bigcup_n B_n\neq\emptyset\right\}\\
    & =\min\left\{|\cA|:\ \cA\subseteq\overline{I}\land\forall_{\langle B_n\rangle\in\overline{I}^\omega}\, \exists_{A\in\cA}\, A\setminus \bigcup_n B_n\neq\emptyset\right\}.
\end{align*}
Now it is easy to see that each family $\cA\subseteq\overline{I}$ witnessing $\add(I)$ is also a witness for $\add'(I)$. Hence, $\add'(I)\leq\add(I)$.
\end{proof}

\begin{proposition}
\label{add'}\
\begin{itemize}
\item[(a)] $\add'(\cM)=\add(\cM)$;
\item[(b)] $\pp\leq\add'(\cN)=\add'(\sigma\overline{\cN})\leq\add(\cM)$;
\item[(c)] $\add'([2^\omega]^{\leq\omega})=\aleph_1$.
\end{itemize}    
\end{proposition}

\begin{proof}
(a): The inequality $\add'(\cM)\leq\add(\cM)$ follows from Proposition \ref{prop-toprep}, while $\add'(\cM)\geq\add(\cM)$ is \cite[Lemma 3.5]{MR3032424}.

(b): Since $\add'(I)$ depends only on compact members of $I$, $\add'(\cN)=\add'(\sigma\overline{\cN})$. From Proposition \ref{prop-toprep} and \cite[Theorem 3.1]{MR1186905} we get $\add'(\sigma\overline{\cN})\leq\add(\sigma\overline{\cN})=\add(\cM)$, while $\pp\leq\add'(\cN)$ is shown in \cite[Theorem 22G]{Fremlin}.

(c): It follows from Proposition \ref{prop-toprep} and the fact that $\add([2^\omega]^{\leq\omega})=\omega_1$.
\end{proof}

\begin{definition}
\label{def-invariant-over-basic}
We say that a $\sigma$-ideal $I$ on $2^\omega$ is invariant over basic sets if given any $s\in 2^{<\omega}$ we have:
$$A\in I\ \Leftrightarrow\ f_s[A]\in I,$$
where $f_s:2^\omega\to V_s$ is given by $f_s(x)=s^\frown x$ and $V_s=\{x\in 2^\omega: x|\lh(s)=s\}$ is the basic clopen set associated to $s$ (here $\lh(s)$ is the length of $s$).
\end{definition}

\begin{proposition}
\label{invariantoverbasic}
The $\sigma$-ideals $\cM$, $\cN$, $\sigma\overline{\cN}$ and $[2^\omega]^{<\omega}$ are invariant over basic sets.     
\end{proposition}

\begin{proof}
In the case of $[2^\omega]^{<\omega}$, it suffices to observe that each $f_s$ is a bijection (so $|A|=|f_s[A]|$). Since each $f_s$ is also a homeomorphism, $\cM$ is invariant over basic sets. Finally, since for each measurable $A\subseteq 2^\omega$ the measure of $f_s[A]$ is equal to the measure of $A$ multiplied by $\frac{1}{2^{\lh(s)}}$, $\cN$ and $\sigma\overline{\cN}$ are invariant over basic sets.
\end{proof}

Finally, we are ready to calculate the invariants $\bb(\I)$ in the case of ideals with topological representation. 

\begin{theorem}
\label{thm-toprep}
Let $I$ be an invariant over basic sets $\sigma$-ideal on $2^\omega$ generated by some family of compact sets. Then:  
$$\min\{\add'(I),\bb\}\leq\bb(\J_{I})\leq\add'(I).$$
Moreover, if $\J_I$ is analytic, then: 
$$\bb(\J_{I})=\min\{\add'(I),\bb\}.$$
\end{theorem}

\begin{proof}
The case of analytic ideals will follow from the general case. Indeed, each analytic ideal with topological representation $\J_{I}$ is ${\bf{\Pi^0_3}}$ (by \cite[Theorem 1.4]{KwelaSabok}), so using \cite[Corollary 6.9]{nasza} we get $\bb(\J_{I})\leq\bb$. 

By \cite[Proposition 2.1]{KwelaSabok}, without loss of generality we may assume that $\J_I$ is an ideal on $D=\{x\in 2^\omega: \exists_{k\in\omega}\, \forall_{n\geq k}\, x(n)=0\}$. Note that $f_s[D]=D\cap V_s$, where $V_s$ and $f_s$ are as in Definition \ref{def-invariant-over-basic}.  

We start with $\bb(\J_{I})\geq\min\{\add'(I),\bb\}$. Fix any $\kappa<\min\{\add'(I),\bb\}$. We need to show that $\kappa<\bb(\J_{I})$. Let $\{\langle E^\alpha_n:n\in\omega\rangle: \alpha<\kappa\}\subseteq \widehat{\cP}_{\J_I}$ be arbitrary. Since $\omega\cdot\kappa<\add'(I)$ (by Proposition \ref{prop-toprep}), there is $\langle B_n:n\in\omega\rangle\in\overline{I}^\omega$ such that for each $n\in\omega$ and $\alpha<\kappa$ there exists $f_\alpha(n)\in\omega$ with $\cl_{2^\omega}(E^\alpha_n)\subseteq B_{f_\alpha(n)}$. 

Since $\kappa<\bb$, there exists $g\in\omega^\omega$ such that $f_\alpha\leq^\star g$, for all $\alpha<\kappa$. Without loss of generality, we may assume that $g$ is strictly increasing.

Let $\{q_n:\ n\in\omega\}$ be an enumeration of $D$ and define $C_0=\bigcup_{i\leq g(1)}(B_i\cap D)\cup\{q_0\}$ and $C_{n}=(\bigcup_{i\leq g(n+1)}(B_{i}\cap D)\cup\{q_{n}\})\setminus\bigcup_{i<n}C_i$. Then $\langle C_n: n\in\omega\rangle\in\cP_{\J_I}$, since: 
$$\cl_{2^\omega}(C_n)\subseteq \cl_{2^\omega}\left(\bigcup_{i\leq g(n+1)}B_i\cup\{q_n\}\right)=\bigcup_{i\leq g(n+1)}B_i\cup\{q_n\}\in I.$$

Fix now $\alpha<\kappa$. We need to show that $\bigcup_{n\in\omega}\left(C_n\cap\bigcup_{i\leq n}E^\alpha_i\right)\in\J_I$.

Let $k_\alpha\in\omega$ be such that $f_{\alpha}(m)\leq g(m)$, for every $m>k_\alpha$. Observe that for each $m>k_\alpha$ we get:
$$E^\alpha_m\subseteq \cl_{2^\omega}(E^\alpha_m)\cap D\subseteq B_{f_\alpha(m)}\cap D\subseteq \bigcup_{n\leq g(m)}B_n\cap D\subseteq\bigcup_{n\leq m-1}C_n.$$
and consequently: 
$$E^\alpha_m\cap\bigcup_{n\in\omega}\left(C_n\cap\bigcup_{i\leq n}E^\alpha_i\right)
\subseteq E^\alpha_m\cap\bigcup_{n\geq m}C_n=\emptyset,$$
since $\langle E^\alpha_n:n\in\omega\rangle\in\widehat{\cP}_{\J_I}$.
Hence, 
$$\bigcup_{n\in\omega}\left(C_n\cap\bigcup_{i\leq n}E^\alpha_i\right)\subseteq\left(\bigcup_{m\in\omega}E^\alpha_m\right)\cap\left(\bigcup_{n\in\omega}\left(C_n\cap\bigcup_{i\leq n}E^\alpha_i\right)\right)\subseteq\bigcup_{m\leq k_\alpha}E^\alpha_m\in\J_I.$$

Now we move to $\bb(\J_I)\leq\add'(I)$. Let $\cA\subseteq\overline{I}$ be the family of cardinality $\add'(I)$ such that for each $\langle B_n:n\in\omega\rangle\in\overline{I}^\omega$ one can find $A\in\cA$ with $A\setminus B_n\neq\emptyset$, for all $n$. 

Without loss of generality we may assume that $\cl_{2^\omega}(A\cap D)=A$ for each $A\in\cA$. Indeed, given $A\in\cA$, fix any countable $E\subseteq A$ such that $\cl_{2^\omega}(E)=A$ (this is possible as $A$ is a closed subset of $2^\omega$) and enumerate $E=\{e_i: i\in\omega\}$. For each $i\in\omega$ find a sequence $\langle x_{i,j}:j\in\omega\rangle\in [D\cap B(e_i,\frac{1}{2^i})]^\omega$ converging to $e_i$ (here $B(e_i,\frac{1}{2^i})$ denotes the open ball in $2^\omega$ of radius $\frac{1}{2^i}$ centered at $e_i$). By defining $\hat{A}=A\cup\{x_{i,j}: i,j\in\omega\}$ we get a closed set such that $\hat{A}=\cl_{2^\omega}(\hat{A}\cap D)$. Moreover, as $A\subseteq\hat{A}$, the family $\{\hat{A}: A\in\cA\}$ satisfies the same property as $\cA$. Thus, we will assume from now on that $\cl_{2^\omega}(A\cap D)=A$, for every $A\in \cA$.

Fix any bijection $h:\omega\to 2^{<\omega}$ and define $E^A_n=f_{h(n)}[A\cap D]$ for each $n\in\omega$ and $A\in\cA$. Observe that:
$$\cl_{2^\omega}(E^A_n)=\cl_{2^\omega}(f_{h(n)}[A\cap D])\subseteq f_{h(n)}[\cl_{2^\omega}(A\cap D)]=f_{h(n)}[A]\in I,$$
since $f_{h(n)}$ is a homeomorphism and $I$ is invariant over basic sets. Hence, each $E^A_n$ belongs to $\J_I$. We claim that the family $\{\langle \hat{E}^A_n: n\in\omega\rangle:A\in\cA\}\subseteq\widehat{\cP}_{\J_I}$, where $\hat{E}^A_n=E^A_n\setminus\bigcup_{i<n}E^A_i$ for all $n\in\omega$ and $A\in\cA$, witnesses $\bb(\J_I)\leq\add'(I)$. 

Fix any partition $\langle B_n:n\in\omega\rangle\in\cP_{\J_I}$ of $D$. 

Note that given any $n\in\omega$, we have $V_{h(n)}\cap\bigcup_{i<n}\cl_{2^\omega}(B_i)\subseteq \bigcup_{i<n}\cl_{2^\omega}(B_i)\in I$. Since $I$ is invariant over basic sets and $f_{h(n)}$ is a bijection, the set $f_{h(n)}^{-1}[V_{h(n)}\cap\bigcup_{i<n}\cl_{2^\omega}(B_i)]$ belongs to $\overline{I}$. Hence, there is $A\in\cA$ such that: 
$$A\setminus \left(f_{h(n)}^{-1}\left[V_{h(n)}\cap\bigcup_{i<n}\cl_{2^\omega}(B_i)\right]\right)\neq\emptyset$$
for all $n\in\omega$. To finish the proof, we will show that $B=\bigcup_{n\in\omega}\left(B_n\cap\bigcup_{i\leq n}\hat{E}^A_n\right)$ is dense in $2^\omega$ (hence, $\cl_{2^\omega}(B)=2^\omega\notin I$ and $B\notin\J_I$). 

Fix any $s\in 2^{<\omega}$. We need to find an element of $B$ belonging to $V_s$. Since
$$\cl_{2^\omega}(A\cap D)\setminus f_s^{-1}\left[V_s\cap\bigcup_{i<h^{-1}(s)}\cl_{2^\omega}(B_i)\right]=A\setminus \left(f_{s}^{-1}\left[V_s\cap\bigcup_{i<h^{-1}(s)}\cl_{2^\omega}(B_i)\right]\right)\neq\emptyset,$$
and $f_s^{-1}\left[V_s\cap\bigcup_{i<h^{-1}(s)}\cl_{2^\omega}(B_i)\right]$ is closed (as $f_s$ is a homeomorphism), we can find $x_s\in (A\cap D)\setminus \left(f_s^{-1}[V_s\cap\bigcup_{i<h^{-1}(s)}B_i]\right)$. Then, using Remark \ref{rem-zawieranie}, we get: 
\begin{align*}
f_s(x_s) & \in V_s\cap \left(E^A_{h^{-1}(s)}\setminus \bigcup_{i<h^{-1}(s)}B_i\right)
\subseteq V_s\cap\bigcup_{j\leq h^{-1}(s)}\left(\hat{E}^A_j\setminus \bigcup_{i<h^{-1}(s)}B_i\right)\\
& \subseteq V_s\cap\bigcup_{j\leq h^{-1}(s)}\left(\hat{E}^A_j\setminus \bigcup_{i<j}B_i\right)\subseteq V_s\cap B.
\end{align*}
This finishes the proof.
\end{proof}

Next result gives some bounds for three well-known ideals having topological representation: $\nwd=\J_\cM$, $\mz=\J_{\sigma\overline{\cN}}$ and $\ctbl=\J_{[2^\omega]^{\leq\omega}}$.

\begin{corollary}\
\label{cor-toprep}
\begin{itemize}
\item[(a)] $\bb(\nwd)=\add(\cM)$.
\item[(b)] $\pp\leq\bb(\mz)\leq\add(\cM)$.
\item[(c)] $\bb(\ctbl)=\aleph_1$.
\end{itemize}
\end{corollary}

\begin{proof}
Follows from Theorem \ref{thm-toprep}, Proposition \ref{add'} and the fact that $\aleph_1\leq\add(\cM)\leq\bb$ (by \cite[Theorem 3.11]{Bartoszynski}). 
\end{proof}

We end this Section with calculation of $\bb(\I)$ for one more well-known ideal defined on the rationals.

Recall that $\conv$ is the ideal on $\mathbb{Q}\cap[0,1]$ generated by sequences in $\mathbb{Q}\cap[0,1]$ that are convergent in $[0,1]$. This is a ${\bf{\Sigma^0_4}}$ ideal contained in $\nwd$, $\mz$ and $\ctbl$ (see \cite[Subsection 3.4]{MR2777744}).
 
\begin{proposition}
$\bb(\conv)=\aleph_1$.    
\end{proposition}

\begin{proof}
By \cite[Proposition 4.1]{MR3269495}, $\bb(\I)\geq\aleph_1$ for every ideal $\I$, so we only need to show that $\bb(\I)\leq\aleph_1$.

Fix any family $\{A^\alpha_n: n\in\omega,\alpha<\aleph_1\}\subseteq\conv$ such that:
\begin{itemize}
    \item each $A^\alpha_n$ is a sequence in $\mathbb{Q}\cap[0,1]$ convergent in $[0,1]$;
    \item $\lim A^\alpha_n\neq \lim A^\beta_m$, for all $(\alpha,n),(\beta,m)\in\aleph_1\times\omega$ such that $(\alpha,n)\neq(\beta,m)$;
    \item $A^\alpha_n\cap A^\alpha_m=\emptyset$ for all $\alpha<\aleph_1$ and $n,m\in\omega$ such that $n\neq m$.
\end{itemize}
We will show that this family witnesses $\bb(\conv)\leq\aleph_1$.

Let $\langle B_n:n\in\omega\rangle\in\cP_\conv$. Since $\bigcup_{n\in\omega}\cl_{[0,1]}(B_n)$ is countable, there is $\alpha<\aleph_1$ such that $\lim A^\alpha_m\notin\bigcup_{n\in\omega}\cl_{[0,1]}(B_n)$, for all $m\in\omega$. In particular, $A^\alpha_m\cap B_n$ is finite for all $n,m\in\omega$.

Suppose to the contrary that $B=\bigcup_{n\in\omega}\left(B_n\cap \bigcup_{i\leq n}A^\alpha_i\right)\in\I$, i.e., there are sequences $D_0,\ldots,D_m\subseteq\mathbb{Q}\cap[0,1]$ convergent in $[0,1]$ and such that $B\subseteq^\star \bigcup_{i\leq m}D_i$. Find $j\in\omega$ such that $\lim A^\alpha_j\neq\lim D_i$ for all $i\leq m$. Observe that:
$$A^\alpha_j\setminus\bigcup_{n<j}B_n=A^\alpha_j\cap\bigcup_{n\geq j}B_n\subseteq \bigcup_{n\geq j}\left(B_n\cap\bigcup_{i\leq n}A^\alpha_i\right)\subseteq B.$$
Now, since $A^\alpha_j\cap \bigcup_{n<j}B_n$ is finite, we get $A^\alpha_j\subseteq^\star B\subseteq^\star \bigcup_{i\leq m}D_i$. Thus, there is $i\leq m$ with $A^\alpha_j\cap D_i$ infinite. This contradicts $\lim A^\alpha_j\neq\lim D_i$ and finishes the proof.
\end{proof}


\section{An ideal with \texorpdfstring{$\bb(\I)>\dd$}{}}
\label{sec:>d}

In this Section we will show that, consistently, there is an ideal $\I$ with $\bb(\I)>\dd$. It will follow from the following more general result.

\begin{lemma}
\label{kappa}
If $\kappa<\cf(\cc)\leq\uu=\cc = \cc^\kappa$, then there exists an ideal $\I$ such that $\bb(\I)>\kappa$.
\end{lemma}

\begin{proof}
Without loss of generality we may assume that $\kappa$ is infinite (as for finite $\kappa$ this is trivial). Recall that for $\cA\subseteq\cP(\omega)$ we write:
$$\I(\cA)=\left\{B\subseteq\omega: B\subseteq^\star\bigcup\cA'\text{ for some }\cA'\in[\cA]^{<\omega}\right\}.$$
In particular, $\I(\cA)$ is an ideal if and only if $\omega\notin\I(\cA)$.

Since $|\cP(\omega)^{\kappa\times\omega}|=\cc^\kappa=\cc$, we can fix an enumeration $\{h_\alpha: \alpha<\cc\}$ of $\cP(\omega)^{\kappa\times\omega}$ such that for each $h\in\cP(\omega)^{\kappa\times\omega}$ and $\alpha<\cc$ there is $\alpha<\beta<\cc$ with $h_\beta=h$.

We will recursively define a sequence $\{\I_\alpha: \alpha<\cc\}$ of subsets of $\cP(\omega)$ such that for each $\alpha<\cc$:
\begin{itemize}
\item[(i)] $\I_\alpha$ is an ideal;
\item[(ii)] $\I_\alpha$ is generated by at most $|\alpha\cdot\kappa|$ sets;
\item[(iii)] if $\beta<\alpha$, then $\I_\beta\subseteq\I_\alpha$;
\item[(iv)] if $h_\alpha[\kappa\times\omega]=\{h_\alpha(\gamma,n): \gamma<\kappa,n\in\omega\}\subseteq\bigcup_{\delta<\alpha}\I_\delta$, then there is $\langle B_n:n\in\omega\rangle\in\cP_{\I_{\alpha}}$ such that $\bigcup_{n\in\omega}\left(B_n\cap\bigcup_{i\leq n}h_\alpha(\gamma,i)\right)\in\I_\alpha$, for all $\gamma<\kappa$.
\end{itemize}

Assume that $\alpha<\cc$ and $\I_\beta$, for all $\beta<\alpha$, are already defined. Observe that $\hat{\I}_\alpha=\bigcup_{\delta<\alpha}\I_\delta$ is an ideal (as a union of an increasing sequence of ideals) generated by at most $|\alpha\cdot(\alpha\cdot\kappa)|=|\alpha\cdot\kappa|$ sets (since $\kappa$ is infinite). Hence, if $h_\alpha(\gamma,n)\notin\hat{\I}_\alpha$ for some $(\gamma,n)\in\kappa\times\omega$, then just put $\I_\alpha=\hat{\I}_\alpha$ and observe that $\I_\alpha$ is as needed. 

Assume now that $h_\alpha[\kappa\times\omega]\subseteq\hat{\I}_\alpha$. Recursively define a sequence $\langle C_n:n\in\omega\rangle$ of subsets of $\omega$ such that $C_n\notin\I(\hat{\I}_\alpha\cup\{C_j: j<n\})$ and $\I(\hat{\I}_\alpha\cup\{C_j: j\leq n\})$ is an ideal. This is possible as $|\alpha\cdot(\alpha\cdot\kappa)+n|<\cc=\uu$ guarantees that $\I(\hat{\I}_\alpha\cup\{C_j: j<n\})$ is not maximal. Next, for each $n\in\omega$ put $B_n=(\{n\}\cup C_n)\setminus\bigcup_{j<n}B_j$ and note that $\langle B_n:n\in\omega\rangle$ is a partition of $\omega$ and $B_n\notin\I(\hat{\I}_\alpha\cup\{B_j: j<n\})=\I(\hat{\I}_\alpha\cup\{C_j: j<n\})$, for each $n$. Moreover, $\I(\hat{\I}_\alpha\cup\{B_k:\ k\in\omega\})$ is an ideal, as otherwise we would have: 
$$\omega\in\I(\hat{\I}_\alpha\cup\{B_k: k\in\omega\})=\bigcup_{n\in\omega}\I(\hat{\I}_\alpha\cup\{B_k: k<n\})=\bigcup_{n\in\omega}\I(\hat{\I}_\alpha\cup\{C_k: k<n\}),$$
which contradicts the fact that $\omega\notin\I(\hat{\I}_\alpha\cup\{C_k: k\leq n\})$ for each $n$.

Define: 
$$\I_\alpha=\I\left(\hat{\I}_\alpha\cup\{B_n: n\in\omega\}\cup\left\{\bigcup_{n\in\omega}\left(B_n\cap\bigcup_{i\leq n}h_\alpha(\gamma,i)\right):\gamma<\kappa\right\}\right).$$
Obviously, items (iii) and (iv) are satisfied. What is more, item (ii) is also satisfied as $|\alpha\cdot\kappa+\omega+\kappa|=|\alpha\cdot\kappa|$. Thus, we only need to check that $\omega\notin\I_\alpha$. 

Assume to the contrary that $\omega\in\I_\alpha$. Then there are $k,m\in\omega$, $\gamma_0,\ldots,\gamma_m<\kappa$ and $A\in\hat{\I}_\alpha$ such that: 
$$\omega=A\cup\bigcup_{j<k}B_j\cup\bigcup_{j\leq m}\left(\bigcup_{n\in\omega}\left(B_n\cap\bigcup_{i\leq n}h_\alpha(\gamma_j,i)\right)\right).$$ 
However, since $\langle B_n:n\in\omega\rangle$ is a partition of $\omega$, we have: 
\begin{align*}
    B_k & =B_k\cap\omega=B_k\cap\left(A\cup\bigcup_{j<k}B_j\cup\bigcup_{j\leq m}\left(\bigcup_{n\in\omega}\left(B_n\cap\bigcup_{i\leq n}h_\alpha(\gamma_j,i)\right)\right)\right) \\
    & \subseteq A\cup\bigcup_{j\leq m}\bigcup_{i\leq k}h_\alpha(\gamma_j,i)\in\hat{\I}_\alpha,
\end{align*}
which contradicts $B_k\notin\I(\hat{\I}_\alpha\cup\{B_j:\ j<k\})$. Thus, $\omega\notin\I_\alpha$.

Once the recursive construction of $\{\I_\alpha: \alpha<\cc\}$ is completed, define $\I=\bigcup_{\alpha<\cc}\I_\alpha$. Clearly, $\I$ is an ideal (by items (i) and (iii)). To finish the proof we need to show that $\bb(\I)>\kappa$. Fix any $\{\langle E^\gamma_n:n\in\omega\rangle: \gamma<\kappa\}\subseteq\widehat{\cP}_\I$. Then for each $(\gamma,n)\in\kappa\times\omega$ there is $\beta(\gamma,n)<\cc$ such that $E^\gamma_n\in\I_{\beta(\gamma,n)}$. Denote $\beta=\sup\{\beta(\gamma,n): (\gamma,n)\in\kappa\times\omega\}$. Since $|\kappa\times\omega|=\kappa<\cf(\cc)$, we get that $\beta<\cc$. Thus, we can find $\beta<\alpha<\cc$ with $h_\alpha(\gamma,n)=E^\gamma_n$, for all $(\gamma,n)\in\kappa\times\omega$. By item (iii), $h_\alpha[\kappa\times\omega]\subseteq\bigcup_{\delta<\alpha}\I_\delta$. Using item (iv) for $\alpha$ we get $\langle B_n:n\in\omega\rangle\in\cP_{\I_{\alpha}}\subseteq\cP_\I$ such that $\bigcup_{n\in\omega}\left(B_n\cap\bigcup_{i\leq n}E^\gamma_i\right)\in\I_\alpha\subseteq\I$, for all $\gamma<\kappa$. Therefore, $\bb(\I)>\kappa$ and the proof is completed.
\end{proof}

\begin{theorem}
If $\dd<\cf(\cc)\leq\uu=\cc = \cc^\dd$, then there exists an ideal $\I$ with $\bb(\I)>\dd$. In particular, it is consistent that there exists an ideal $\I$ such that $\bb(\I)>\dd$.
\end{theorem}

\begin{proof}
The first part follows directly from Lemma \ref{kappa}.

We will show that the assumptions $\dd<\cf(\cc)\leq\uu=\cc = \cc^\dd$ are consistent (specifically, that they hold in the generic extension obtained by adding $\aleph_2$ random reals to the model of GCH).

Suppose that GCH is satisfied in the ground model $V$.
Let 
$$M(\aleph_2) = \{C\subseteq 2^{\aleph_2}:\text{$C$ is a closed set of positive measure}\}$$
(ordered by the inclusion modulo null sets) be a forcing which adds $\aleph_2$ random reals (here we consider the product measure on $2^{\aleph_2}$).
Let $G$ denote a $M(\aleph_2)$-generic filter.

It is known that $\dnumber=\aleph_1$ and $\unumber=\continuum=\aleph_2$ in $V[G]$ (see e.g.~\cite[p.~474]{MR2768685}), so $\dd<\cf(\cc)=\uu=\cc$. It remains to show that  $\cc = \cc^\dd$ in $V[G]$.

Denote $\kappa = |M(\aleph_2)|$, $\lambda=\aleph_1$ and $\delta = \kappa^{\lambda}$. Since GCH holds in $V$, 
we have 
$\aleph_2 = 2^{\aleph_1}$,
and consequently
$$\kappa = |M(\aleph_2)| = |[\aleph_2]^{<\aleph_0}|^{\aleph_0}  =\aleph_2^{\aleph_0} = (2^{\aleph_1})^{\aleph_0} = 2^{\aleph_1}=\aleph_2$$
and
$$\delta = \aleph_2^{\aleph_1} = (2^{\aleph_1})^{\aleph_1} = \aleph_2.$$

Hence, $\kappa, \lambda,\delta$ are infinite cardinals. As $M(\aleph_2)$ is ccc, using \cite[Lemma~IV.3.11 at p.~267]{MR2905394} we obtain in $V[G]$ the inequality $2^\lambda\leq \delta$ which yields:
$$\continuum^\dnumber = (2^{\aleph_0})^{\dnumber}  = 2^\dnumber = 2^{\aleph_1} = 2^\lambda\leq \delta = \aleph_2 = \continuum.$$
Since $\continuum\leq \continuum^\dnumber$, we obtain $\continuum=\continuum^\dnumber$, and the proof is finished. 
\end{proof}


\section{Dominating number and analytic ideals}
\label{sec:analytic}

In this Section we show that $\bb(\I)\leq\dd$ for every analytic ideal $\I$. We will apply ideas developed by Kat\v{e}tov in~\cite{Katetov,Kat1,Kat2} and by Debs and Saint Raymond in~\cite{Debs}. We start with recalling several technical notions. 

If $\I$ and $\J$ are ideals on $X$ and $Y$, respectively, then we say that \emph{$\I$ is below $\J$ in the Kat\v{e}tov order} (in short: $\I\leq_{K}\J$), if there is a function $f:Y\to X$ such that $f^{-1}[A]\in\J$ for every $A\in\I$. If we can find a bijection with the above property, then we say that \emph{$\J$ contains an isomorphic copy of $\I$} and write $\I\sqsubseteq\J$.

If $\{X_t: t\in T\}$ is a family of sets, then $\sum_{t\in T}X_t=\{(t,x): t\in T,x\in X_t\}$ is its disjoint sum. The vertical section of a set $A\subseteq \sum_{t\in T}X_t$ at a point $t\in T$ is defined 
by $A_{(t)} = \{x\in X_t: (t,x)\in A\}$. 

Let $\I$ be an ideal on a countable set $T$ and $\{\I_t: t\in T\}$ be a family of ideals. Following \cite{Katetov}, we define the ideal:
$$\I-\sum_{t\in T}\I_t=\{A\subseteq \sum_{t\in T}\left(\bigcup\I_t\right):\{t\in T: A_{(t)}\notin\I_t\}\in\I\}.$$
In particular, if $\J$ is some ideal on $S$ and $\I_t=\J$ for all $t\in T$, then we denote: 
$$\I\otimes\J=\I-\sum_{t\in T}\I_t=\{A\subseteq T\times S: \{t\in T: A_{(t)}\notin\J\}\in\I\}.$$
Following \cite{Kat2,Kat1}, let $\{\fin^\alpha:1\leq\alpha<\omega_1\}$ be the family of ideals given by:
\[
\fin^\alpha=\begin{cases}
    \fin, & \text{if }\alpha=1,\\
    \fin\otimes\fin^\beta, & \text{if }\alpha=\beta+1,\\
    \I_\alpha-\sum_{\beta<\alpha}\fin^\beta, & \text{if }\alpha\text{ is a limit ordinal},
\end{cases}
\]
where $\I_\alpha$ is the ideal on the ordinal $\alpha$ generated by the family $\{\beta:\beta<\alpha\}$ (in particular, $\I_\omega=\fin$).

Finally, for ideals $\I$, $\J$ and $\K$ we define:
	\begin{equation*}
			\bb(\I,\J,\K) = 
			\min \left\{|\cE|:\cE\subseteq\widehat{\cP}_\K \land \forall_{\langle A_n\rangle\in\cP_\J}\, \exists_{\langle E_n\rangle\in\cE} \, \bigcup_{n\in\omega}\left(A_n\cap \bigcup_{i\leq n}E_i\right)\notin\I\right\}
	\end{equation*}
(the above is an equivalent form of the coefficient defined in \cite{MR3624786} -- the equivalence is shown in \cite[Proposition 3.9]{nasza}). Note that $\bb(\I,\I,\I)=\bb(\I)$. Moreover, it is easy to see that if $\cK\subseteq\cK'$, then $\bb(\I,\J,\K)\geq\bb(\I,\J,\K')$ (see \cite[Proposition 3.8]{nasza}).

We are ready to prove some lemmas, which will imply that $\bb(\I)\leq\dd$ for every analytic ideal $\I$.

\begin{lemma}
\label{successor}
If $1\leq\alpha<\omega_1$ is an ordinal, $\fin^{\alpha+1}\not\sqsubseteq\I$, but $\fin^{\alpha}\sqsubseteq\I$, then $\bb(\I)\leq\dd$.
\end{lemma}

\begin{proof}
If $\alpha=1$ then this is true by \cite[Theorem 4.2(4) and Theorem 4.2(9)]{nasza}.

Let now $1<\alpha<\omega_1$ be any ordinal. We need to standardize the notation so that we will be able to proceed with the proof for both successor and limit ordinals $\alpha$. 

If $\alpha=\beta+1$ is a successor ordinal, put $\gamma=\omega$ and write $\beta_\delta=\beta$ for all $\delta<\gamma$. Note that $\fin^\alpha=\fin\otimes\fin^\beta=\I_\gamma-\sum_{\delta<\gamma}\fin^{\beta_\delta}$. 

On the other hand, if $\alpha$ is a limit ordinal, put $\gamma=\alpha$ and write $\beta_\delta=\delta$ for all $\delta<\gamma$. Note that $\fin^\alpha=\I_\alpha-\sum_{\delta<\gamma}\fin^{\delta}=\I_\gamma-\sum_{\delta<\gamma}\fin^{\beta_\delta}$.

Now we can proceed with the proof not caring if $\alpha$ is a successor or a limit ordinal.

For each $\delta<\gamma$ let $\fin^{\beta_\delta}(\omega)$ be any isomorphic copy of $\fin^{\beta_\delta}$ on $\omega$. Then $\fin^\alpha$ and $\J=\I_\gamma-\sum_{\delta<\gamma}\fin^{\beta_\delta}(\omega)$ are isomorphic, so $\J\sqsubseteq\I$. Thus, without loss of generality we may assume that $\I$ is an ideal on $\gamma\times\omega$ such that $\J\subseteq\I$ (this can be done by considering the ideal $\{f[A]:A\in\I\}\supseteq\J$, instead of $\I$, where $f$ is the bijection witnessing $\J\sqsubseteq\I$). 

We will show that $\bb(\I,\I,\J)\leq\dd$. This will finish the proof as $\bb(\I)\leq\bb(\I,\I,\J)$ (by $\J\subseteq\I$).

Fix a strictly increasing sequence $\langle \gamma_i:i\in\omega\rangle\in\gamma^\omega$ converging to $\gamma$ (in the order topology) such that $\gamma_0=0$. Let $\{f_\alpha:\alpha<\dd\}\subseteq\omega^\omega$ be a $\leq_\fin$-dominating family. Without loss of generality we may assume that each $f_\alpha$ is strictly increasing and satisfies $f_\alpha(0)=0$. Define $A^\alpha_n=(\gamma_{f_\alpha(n+1)}\setminus \gamma_{f_\alpha(n)})\times\omega$ for all $\alpha<\dd$ and $n\in\omega$. Then $\{\langle A^\alpha_n:n\in\omega\rangle:\alpha<\dd\}\subseteq\widehat{\cP}_{\J}$. We claim that this family witnesses $\bb(\I,\I,\J)\leq\dd$. 

Fix any $\langle B_n:n\in\omega\rangle\in\cP_{\I}$. There are two possibilities.

Assume first that there is $X\notin\I$ such that $(X\cap B_n)_{(\delta)}\in\fin^{\beta_\delta}(\omega)$, for all $n\in\omega$ and $\delta<\gamma$. For $\delta<\gamma$ let $m_\delta\in\omega$ be such that $\gamma_{f_0(m_\delta)}\leq\delta<\gamma_{f_0(m_\delta+1)}$ (i.e., $\{\delta\}\times\omega\subseteq A^0_{m_\delta}$). Observe that:
\begin{align*}
    \left(\bigcup_{i\in\omega}\left(A^0_i\setminus \bigcup_{n<i}B_n\right)\right)_{(\delta)} & \supseteq \left(A^0_{m_\delta}\setminus \bigcup_{n<m_\delta}B_n\right)_{(\delta)}\\
    & =\omega\setminus\bigcup_{n<m_\delta}(B_n)_{(\delta)}\\
    & \supseteq X_{(\delta)}\setminus\bigcup_{n<m_\delta}(X\cap B_n)_{(\delta)}.
\end{align*}
Define $Y_\delta=\bigcup_{n<m_\delta}(X\cap B_n)_{(\delta)}\in\fin^{\beta_\delta}(\omega)$ and $Y=\bigcup_{\delta<\gamma}\{\delta\}\times Y_\delta\in\J\subseteq\I$. Then we have:
$$\I\not\ni X\setminus Y\subseteq\bigcup_{i\in\omega}\left(A^0_i\setminus \bigcup_{n<i}B_n\right)\subseteq \bigcup_{n\in\omega}\left(B_n\cap\bigcup_{i\leq n}A^0_i\right)$$ 
(the last inclusion follows from Remark \ref{rem-zawieranie}). 

Assume now that if $X\subseteq\gamma\times\omega$ satisfies $(X\cap B_n)_{(\delta)}\in\fin^{\beta_\delta}(\omega)$ for all $n\in\omega$ and $\delta<\gamma$ then $X\in\I$. Recall that the ideal $\fin\otimes\J$ (which is isomorphic to $\fin^{\alpha+1}$) has three kinds of generators:
\begin{itemize}
    \item sets of the form $\{n\}\times\gamma\times\omega$, for $n\in\omega$ (\emph{generators of the first type});
    \item sets of the form $\bigcup_{n\in\omega}\{n\}\times \gamma_{g(n)}\times\omega$, for $g\in\omega^\omega$ (\emph{generators of the second type});
    \item sets $G\subseteq\omega\times\gamma\times\omega$ such that $G_{(n,\delta)}\in\fin^{\beta_\delta}(\omega)$, for all $(n,\delta)\in\omega\times\gamma$ (\emph{generators of the third type}).
\end{itemize}
Consider the function $h:\gamma\times\omega\to\omega\times\gamma\times\omega$ defined by the formula $h(\delta,j)=(n(\delta,j),\delta,j)$, where $n(\delta,j)\in\omega$ is given by $(\delta,j)\in B_{n(\delta,j)}$. Observe that:
\begin{itemize}
    \item if $G=\{n\}\times\gamma\times\omega$ is a generator of he first type, then $h^{-1}[G]\subseteq B_n\in\I$, for every $n\in\omega$;
    \item if $G\subseteq\omega\times\gamma\times\omega$ is a generator of the third type, then $(h^{-1}[G]\cap B_n)_{(\delta)}\subseteq G_{(n,\delta)}\in\fin^{\beta_\delta}(\omega)$, for all $(n,\delta)\in\omega\times\gamma$, so $h^{-1}[G]\in\I$ by our assumption.
\end{itemize}
On the other hand, $\fin^{\alpha+1}\not\leq_K\I$ (by $\fin^{\alpha+1}\not\sqsubseteq\I$ and \cite[Example 4.1]{Kat}), so there has to be a generator of the second type, $G=\bigcup_{n\in\omega}\{n\}\times \gamma_{g(n)}\times\omega$ for some $g\in\omega^\omega$, such that: 
$$\I\not\ni h^{-1}\left[G\right]\subseteq \bigcup_{n\in\omega}\left(B_n\cap(\gamma_{g(n)}\times\omega)\right).$$ 

Find $\alpha<\dd$ with $g\leq^\star f_\alpha$ and let $m\in\omega$ be such that $g(n)\leq f_\alpha(n)$ for all $n>m$. Note that:
\begin{align*}
    \bigcup_{n\in\omega}\left(B_n\cap\bigcup_{i\leq n}A^\alpha_i\right) & \supseteq\bigcup_{n\in\omega}\left(B_n\cap(\gamma_{f_\alpha(n)}\times\omega)\right)\\
    & \supseteq\bigcup_{n>m}\left(B_n\cap(\gamma_{g(n)}\times\omega)\right)\supseteq h^{-1}[G]\setminus \bigcup_{n\leq m} B_n\notin\I
\end{align*}
(as $\bigcup_{n\leq m} B_n\in\I$). This finishes the proof.
\end{proof}

\begin{lemma}
\label{Fin^alpha generators}
For each $1<\alpha<\omega_1$ there is a family $\{A_{f}:f\in\omega^\omega\}\subseteq\fin^\alpha$ satisfying:
\begin{itemize}
    \item for every $A\in\fin^\alpha$ there is $f\in\omega^\omega$ with $A\subseteq A_f$;
    \item $A_f\subseteq A_g$ whenever $f\leq g$ (i.e., $f(n)\leq g(n)$ for all $n\in\omega$).
\end{itemize}
\end{lemma}

\begin{proof}
We will show it inductively. 

This is clear for $\alpha=2$ as witnessed by the family: 
$$\left\{(f(0)\times\omega)\cup\bigcup_{i\in\omega}\{i\}\times f(i):f\in\omega^\omega\right\}\subseteq\fin^2.$$

Fix now any partition $\langle C_n:n\in\omega\rangle$ of $\omega$ into infinite sets and for each $n$ let $h_n:\omega\to C_n$ be the increasing enumeration of $C_n$.

Assume that for some $1<\alpha<\omega_1$ we have the needed family $\{A_f:f\in\omega^\omega\}\subseteq\fin^\alpha$. We need to find the required family for $\alpha+1$. Given $f\in\omega^\omega$, define:
$$B_f=(f(0)\times \bigcup\fin^\alpha)\cup\bigcup_{n\in\omega}\{n\}\times A_{f\circ h_n}.$$
Again it is clear that $\{B_{f}:f\in\omega^\omega\}\subseteq\fin^{\alpha+1}$ is as needed. 

Finally, assume that $1<\alpha<\omega_1$ is a limit ordinal and for each $\beta<\alpha$ we have the required family $\{A^\beta_{f}:f\in\omega^\omega\}\subseteq\fin^\beta$. Find a bijection $g:\omega\to\alpha$ and a strictly increasing sequence $\langle\beta_i:i\in\omega\rangle$ converging to $\alpha$ (in the order topology). For each $f\in\omega^\omega$, define:
$$B_f=\left(\sum_{\beta<\beta_{f(0)}}\bigcup\fin^\beta\right)\cup\left(\bigcup_{n\in\omega}\{g(n)\}\times A^{g(n)}_{f\circ h_n}\right).$$
Once more it is clear that $\{B_{f}:f\in\omega^\omega\}\subseteq\fin^{\alpha}$ is as needed.
\end{proof}

\begin{lemma}
\label{limit2}
If $\alpha<\omega_1$ is a limit ordinal, $\fin^{\alpha}\not\sqsubseteq\I$, but $\fin^{\beta}\sqsubseteq\I$ for all $\beta<\alpha$, then $\bb(\I)\leq\dd$.
\end{lemma}

\begin{proof}
Find an increasing sequence $\langle\beta_i:i\in\omega\rangle$ converging to $\alpha$ (in the order topology). Since $\fin^{\beta}\sqsubseteq\I$ for all $\beta<\alpha$, we can assume that $\I$ is an ideal on $\omega$ and $\fin^{\beta_i}(\omega)\subseteq\I$ for all $i$, where $\fin^{\beta_i}(\omega)$ is some copy of $\fin^{\beta_i}$ on $\omega$. For each $i\in\omega$ let $\{B^i_{f}:f\in\omega^\omega\}\subseteq\fin^{\beta_i}(\omega)$ be the family from Lemma \ref{Fin^alpha generators}. 

Fix any partition $\langle C_n:n\in\omega\rangle$ of $\omega$ into infinite sets and for each $n$ let $h_n:\omega\to C_n$ be the increasing enumeration of $C_n$. Let $\{f_\alpha:\alpha<\dd\}\subseteq\omega^\omega$ be a $\leq_\fin$-dominating family of strictly increasing functions and recursively define $A^\alpha_n=B^n_{f_\alpha\circ h_n}\setminus\bigcup_{i<n}A^\alpha_i$, for all $\alpha<\dd$ and $n\in\omega$. Then $\{\langle A^\alpha_n:n\in\omega\rangle:\alpha<\dd\}\subseteq\widehat{\cP}_\I$ as $\fin^{\beta_i}(\omega)\subseteq\I$ for all $i$. We will show that $\{A^\alpha_n:n\in\omega,\alpha<\dd\}$ witnesses $\bb(\I)\leq\dd$.

Fix any $\langle B_n: n\in\omega\rangle\in\cP_{\I}$ and consider the function $h:\omega\to\alpha\times\omega$ defined by $h(i)=(\beta_n,i)$, where $n\in\omega$ is given by $i\in B_n$. Since $\I_\alpha-\sum_{\beta<\alpha}\fin^{\beta}(\omega)$ is isomorphic to $\fin^\alpha$ and $\fin^{\alpha}\not\leq_K\I$ (by $\fin^{\alpha}\not\sqsubseteq \I$ and \cite[Example 4.1]{Kat}), there is $Y\in \I_\alpha-\sum_{\beta<\alpha}\fin^{\beta}(\omega)$ such that $h^{-1}[Y]\notin\I$. Actually, for each $\beta<\alpha$ either $h^{-1}[\{\beta\}\times\omega]=\emptyset\in\I$ (if $\beta\neq\beta_i$ for all $i\in\omega$) or $h^{-1}[\{\beta\}\times\omega]\subseteq B_n\in\I$ for some $n\in\omega$ (if $\beta=\beta_n$). Hence, the set:
$$X=h^{-1}\left[\bigcup\left\{\{\delta\}\times Y_{(\delta)}: Y_{(\delta)}\in\fin^\delta(\omega)\right\}\right]$$
does not belong to $\I$, while $h[X]\subseteq h[h^{-1}[Y]]\subseteq Y\in \I_\alpha-\sum_{\beta<\alpha}\fin^{\beta}(\omega)$. Observe that $X\cap B_n\subseteq (h[X])_{(\beta_n)}\in\fin^{\beta_n}$ for all $n\in\omega$. 

Then for each $n$ there is $g_n\in\omega^\omega$ such that $X\cap B_n\subseteq B^n_{g_n}$. Define $g\in\omega^\omega$ by $g\restriction C_n=g_n\circ h_n^{-1}$. Find $\alpha<\dd$ with $g\leq^\star f_\alpha$ and let $m\in\omega$ be such that $g(i)\leq f_\alpha(i)$ whenever $i\geq m$. Let $k\in\omega$ be maximal such that $C_k\cap m\neq\emptyset$. Observe that for $n>k$ and any $j\in\omega$ we have:
$$(f_\alpha\circ h_n)(j)=f_\alpha(h_n(j))\geq g(h_n(j))=g_n(j).$$
Hence, $X\cap B_n\subseteq B^n_{g_n}\subseteq B^n_{f_\alpha\circ h_n}\subseteq\bigcup_{i\leq n}A^\alpha_i$, whenever $n>k$. Thus, 
$$\bigcup_{n\in\omega}\left(B_n\cap\bigcup_{i\leq n}A^\alpha_i\right)\supseteq\bigcup_{n>k}\left(B_n\cap \bigcup_{i\leq n}A^\alpha_i\right)\supseteq\bigcup_{n>k}\left(X\cap B_n\right)\notin\I$$
as $\I\not\ni X=\bigcup_{n\in\omega}\left(X\cap B_n\right)$ and $\bigcup_{n\leq k}B_n\in\I$.
\end{proof}

\begin{theorem}
\label{Fin^alpha}
If there is $\alpha<\omega_1$ such that $\fin^\alpha\not\sqsubseteq\I$, then $\bb(\I)\leq\dd$.
\end{theorem}

\begin{proof}
Let $\alpha<\omega_1$ be minimal such that $\fin^{\alpha}\not\sqsubseteq\I$. If $\alpha$ is a successor ordinal, use Lemma \ref{successor}. On the other hand, if $\alpha$ is a limit ordinal, we can apply Lemma \ref{limit2}.
\end{proof}

Recall that for an ideal $\I$ on $X$ we write $\I^\star=\{A\subseteq X: X\setminus A\in\I\}$ and call it the \emph{dual filter of $\I$}. 

\begin{corollary}
\label{rk}
If $\I$ is Borel-separated from $\I^\star$ (i.e., there is a Borel set $S\subseteq 2^\omega$ such that $\I\subseteq S$ and $\I^\star\cap S=\emptyset$), then $\bb(\I)\leq\dd$. In particular, $\bb(\I)\leq\dd$ for every analytic ideal $\I$.
\end{corollary}

\begin{proof}
We claim that there is $\alpha<\omega_1$ such that $\fin^\alpha\not\sqsubseteq\I$. Indeed, by \cite[Theorem 3.2, Theorem 6.5 and Lemma 7.2]{Debs}, if $\fin^\alpha\sqsubseteq\I$, then $\I$ is not $\bf{\Delta^0_{1+\alpha}}$-separated from $\I^\star$. Thus, the assumption that $\fin^\alpha\sqsubseteq\I$, for all $\alpha<\omega_1$, would contradict the fact that $\I$ is Borel-separated from $\I^\star$. Now it suffices to apply Theorem \ref{Fin^alpha}.

For each analytic ideal, its dual filter is also analytic (by \cite[Subsection 1.1]{Debs}), so they can be Borel-separated (by the Lusin separation theorem). This proves the "In particular" part.
\end{proof}

Actually, we do not know whether Corollary \ref{rk} can be strengthened by replacing $\dd$ with $\bb$. In particular, we do not know the answer to the following problem.

\begin{problem}
Does there consistently exist an analytic ideal $\I$ with $\bb(\I)>\bb$?
\end{problem}


%


\bibliographystyle{amsplain}
\bibliography{b-references}

\end{document}